\documentclass[12pt,leqno]{article}
\usepackage{amsthm,amsfonts,amssymb,amsmath,eufrak,oldgerm}
\usepackage{epsfig}
\numberwithin{equation}{section} 

\renewcommand\d{\partial}

\renewcommand\b{\beta}

\renewcommand\o{\omega}
\newcommand\R{\mathbb R}
\newcommand\C{\mathbb C}

\def\eps{\varepsilon}
\def\e{\varepsilon}

\newcommand\kernel{\hbox{\rm Ker}}

\newcommand\br{\begin{remark}}
\newcommand\er{\end{remark}}
\newcommand\bp{\begin{pmatrix}}
\newcommand\ep{\end{pmatrix}}
\newcommand\be{\begin{equation}}
\newcommand\ee{\end{equation}}
\newcommand\ba{\begin{equation}\begin{aligned}}
\newcommand\ea{\end{aligned}\end{equation}}

\newcommand{\bap}{\begin{app}}
\newcommand{\eap}{\end{app}}
\newcommand{\begs}{\begin{exams}}
\newcommand{\eegs}{\end{exams}}
\newcommand{\beg}{\begin{example}}
\newcommand{\eeg}{\end{exaplem}}
\newcommand{\bpr}{\begin{proposition}}
\newcommand{\epr}{\end{proposition}}
\newcommand{\bt}{\begin{theorem}}
\newcommand{\et}{\end{theorem}}
\newcommand{\bc}{\begin{corollary}}
\newcommand{\ec}{\end{corollary}}
\newcommand{\bl}{\begin{lemma}}
\newcommand{\el}{\end{lemma}}
\newcommand{\bd}{\begin{definition}}
\newcommand{\ed}{\end{definition}}
\newcommand{\brs}{\begin{remarks}}
\newcommand{\ers}{\end{remarks}}

\newcommand{\BbbR}{{\mathbb R}}

\newcommand{\Id}{{\rm Id }}

\newcommand{\Res}{{\rm Residue}}

\newtheorem{theorem}{Theorem}[section]
\newtheorem{proposition}[theorem]{Proposition}
\newtheorem{corollary}[theorem]{Corollary}
\newtheorem{lemma}[theorem]{Lemma}
\newtheorem{definition}[theorem]{Definition}

\newtheorem{example}[theorem]{Example}
\newtheorem{remark}[theorem]{Remark}

\newtheorem{exams}[theorem]{Examples}

\newcommand\cB{{\cal  B}}

\newcommand\cW{{\cal  W}}

\newcommand\cG{{\cal  G}}

\newcommand\cN{{\cal  N}}
\newcommand\cE{{\cal  E}}

\newcommand\cM{{\mathcal M}}
\newcommand\cT{{\mathcal T}}

\newcommand\cZ{{\cal  Z}}

\newcommand\tG{{\tilde G}}

\title{
Conditional stability of unstable 
viscous shock waves in
compressible gas dynamics and MHD
}

\author{\sc \small 
Kevin Zumbrun\thanks{Indiana University, Bloomington, IN 47405;
kzumbrun@indiana.edu:
Research of K.Z. was partially supported
under NSF grants no. DMS-0300487 and DMS-0801745.
 }}
\begin{document}

\maketitle

%%%%%%%%%%%%%%%%%%%%%%%%%%%%%%%%%%%%%%%%%%%%%%%%%%%%%%%%%%%%%%%%%%%%%%%%%%%%%%%%%%%%%%%%%%%%%

\begin{abstract}
Extending our previous work in the strictly parabolic case,
we show that a linearly unstable Lax-type viscous shock solution
of a general quasilinear hyperbolic--parabolic system of conservation laws
possesses a translation-invariant center stable manifold within which
it is nonlinearly orbitally stable with respect to small $L^1\cap H^3$ 
perturbations, converging time-asymptotically
to a translate of the unperturbed wave.
That is, for a shock with $p$ unstable eigenvalues, we establish
conditional stability on a codimension-$p$ manifold of initial data,
with sharp rates of decay in all $L^p$.
For $p=0$, we recover the result of unconditional stability obtained
by Mascia and Zumbrun. 
The main new difficulty in the hyperbolic--parabolic
case is to construct an invariant manifold in the absence of parabolic
smoothing.
%despite apparent loss of regularity in the usual fixed-point mapping
%by which this is done.
\end{abstract}

%\clearpage
\tableofcontents
%\clearpage
%%%%%%%%%%%%%%%%%
\bigbreak
\section{Introduction}
In this paper, extending our previous work in the semilinear and
quasilinear parabolic case \cite{Z5,Z6},
we study for general quasilinear hyperbolic--parabolic systems of
conservation laws including the equations of compressible gas dynamics
and MHD 
%the conditional stability of a linearly unstable viscous Lax shock.
the conditional stability and existence of a center stable manifold
of a linearly unstable viscous Lax shock.
This is part of a larger program initiated in 
\cite{TZ1,TZ2,TZ3,TZ4,SS,BeSZ,Z5}
going beyond simple stability analysis to study nontrivial dynamics
and bifurcation, and
associated physical phenomena, of perturbed viscous shock waves in
the presence of {linear instability}.
%THIS?
As discussed for example in \cite{AMPZ,GZ,Z7}, such conditionally
stable shock waves can play an important role in asymptotic
behavior as metastable states, 
%whose center stable manifolds
%separate basins of attraction of nearby stable solutions.
and their center stable manifolds as separatrices 
bounding the basins of attraction of nearby stable solutions.

The main issue in the present case is to construct a (translation-invariant) 
center stable manifold about a standing viscous shock wave;
once this is done, the conditional stability analysis follows 
in straightforward fashion by a combination of
the semilinear argument of \cite{Z5} and the unconditional
stability analysis of \cite{MaZ2,MaZ3,Z2} in the hyperbolic--parabolic case.
As discussed in \cite{Z6}, the technical difficulty
in constructing the center stable manifold for quasilinear equations 
%of hyperbolic or partly hyperbolic type 
is 
%that the associated linearized semigroup does not exhibit smoothing 
%as in the parabolic case, and so there is 
an apparent loss of regularity in the usual fixed point argument by 
which the center stable manifold is constructed.
See also the related discussion of \cite{LPS1,LPS2}.
% in the quasilinear parabolic case.
We accomplish this by the method introduced in \cite{Z6} in the quasilinear 
parabolic case,
combining an implicit fixed-point scheme with suitable time-weighted 
$H^s$-energy estimates.
For related energy estimates, see \cite{TZ3}.

%\medbreak

Consider a viscous shock solution 
\be\label{prof}
U(x,t)=\bar U(x-st), \qquad
\lim_{z\to \pm \infty}\bar U(z)=U_\pm,
\ee
of a hyperbolic--parabolic system of conservation laws
\begin{equation} \label{cons}
\begin{aligned}
U_t + F(U)_x &= (B(U)U_x)_x, \\
\end{aligned}
\end{equation}
$x \in \mathbb{R}$, $U$, $F \in \mathbb{R}^n$, $B\in\mathbb{R}^{n
\times n}$.
Profile $\bar U$ satisfies the traveling-wave ODE
\be\label{ode}
B(U)U'=F( U)- F(U_-)- s(U-U_-).
\ee
Denote $A( U)=F_U( U)$,
\be\label{pmvalues}
B_\pm:=\lim_{z\to \pm \infty} B(z)=B(U_\pm),
\qquad
A_\pm:=\lim_{z\to \pm \infty} A(z)=F_U(U_\pm).
\ee

Following \cite{TZ3,Z2,Z3}, we make the structural assumptions:
\medskip

(A1) \quad
$U=\left(\begin{array}{c} U_1 \\ U_2\end{array}\right)$,
\quad $F=\left(\begin{array}{c} F_1 \\ F_2\end{array}\right)$,
\quad $B=\left(\begin{array}{cc} 0 & 0 \\ 0 & b \end{array}\right)$, \quad
$b$ nonsingular, where $U\in \BbbR^n$, $U_1\in  \BbbR^{n-r}$, $U_2\in \BbbR^r$, and
$b\in \BbbR^{r\times r}$;
moreover, the $U_1$-coordinate $F_1( U)$ of $F$ is {\it linear in} $U$
(strong block structure).

(A2)\quad There exists a smooth, positive definite 
matrix $A^0( U)$, without loss of generality block-diagonal,
such that $A^0_{11} A_{11}$ is symmetric,
$A^0_{22}b$ is positive definite but not necessarily
symmetric, and $(A^0A)_\pm$ is symmetric.

(A3)\quad 
No eigenvector of $A_\pm$ lies in $\kernel B_\pm$
(genuine coupling \cite{Kaw,KSh}).
\medskip

To (A1)--(A3), we add the following more detailed hypotheses.
Here and elsewhere, $\sigma(M)$ denotes the spectrum of a matrix
or linear operator $M$.
\medbreak

(H0) \quad  $F, B\in C^{k}$, $k\ge 4$.

(H1) \quad  $\sigma(A_{11})$ 
real, constant multiplicity, with 
$\sigma(A_{11})<s$ or $\sigma(A_{11})>s$. 

(H2) \quad $\sigma(A_\pm)$ real, simple, and different from $s$.

\medbreak

Conditions (A1)--(A3) 
%(or the alternatives described in Remark \ref{genen})
are a slightly strengthened version of the
corresponding hypotheses of \cite{MaZ4,Z2,Z3} for general systems with ``real'', or
partially parabolic viscosity, the difference lying in the strengthened
block structure condition (A1): in particular, the assumed linearity
of the $U_1$ equation.  Conditions (H0)--(H2) are
the same as those in \cite{MaZ4,Z2,Z3}.
The class of equations satisfying our assumptions, though not complete, is sufficiently
broad to include many models of physical interest,
in particular compressible Navier--Stokes equations and the equations of compressible
magnetohydrodynamics (MHD), expressed
in Lagrangian coordinates, with either ideal
or ``real'' van der Waals-type equation of state as described
in Appendix \ref{examples}.
See \cite{TZ3} for further discussion/generalization.

\br\label{laxrmk}
\textup{
Conditions (H1)--(H2) imply that $U_\pm$ are nonhyperbolic rest points of
ODE (\ref{ode}) expressed in terms of the $U_2$-coordinate, whence, by standard ODE theory,
\be\label{profdecay}
|\partial_x^r (\bar U-U_\pm)(x)|\le Ce^{-\eta|x|},
\qquad
0\le r\le k+1,
\ee
for $x\gtrless 0$, some $\eta$, $C>0$; in particular, $|\bar U'(x)| \le
Ce^{-\eta|x|}$.
}
\er

%\medbreak
By the change of coordinates $x\to x-st$, we may assume without
loss of generality $s=0$ as we shall do from now on.
Linearizing \eqref{cons} about 
%$\bar U$ yields linearized equations
the (now stationary) solution $U\equiv \bar U$ yields linearized equations
\begin{equation} \label{L}
U_t=LU:=-(AU)_x+(BU_x)_x,
\end{equation}
\begin{equation}
B(x):= B(\bar U(x)), \quad
A(x)V:=
dF(\bar U(x))V-(dB(\bar U(x))V)\bar U_x,
\label{AandBov}
\end{equation}
for which the generator $L$ possesses \cite{He,Sat,ZH} 
both a translational zero-eigenvalue
and essential spectrum tangent at zero to the imaginary axis. 

Our first main result is the existence of a translation-invariant center stable
manifold about $\bar U$.
Define the mixed norm
\be\label{mixed}
|U|_{H^{1,2}}:=
\sqrt{|U_1|_{H^{1}}^2+ |U_2|_{H^{2}}^2}
\ee
suggested by the hyperbolic--parabolic structure (A1).

\bt\label{t:maincs}
Under assumptions (A1)--(A3), (H0)--(H2), 
there exists in an $H^{1,2}$
neighborhood of the set of translates of $\bar u$
a codimension-$p$ translation invariant Lipschitz (with respect to $H^{1,2}$) 
center stable manifold $\cM_{cs}$,
tangent to quadratic order at $\bar U$ to the center stable subspace 
$\Sigma_{cs}$ of $L$ in the sense that
\be\label{tanbd1}
|\Pi_{u}(U-\bar U)|_{H^{1,2}}\le C |\Pi_{cs}(U-\bar U)|_{H^{1,2}}^2
\ee
for $U\in \cM_{cs}$ where $\Pi_{cs}$ and $\Pi_u$ denote the center-stable
and stable eigenprojections of $L$,
that is (locally) invariant under the forward time-evolution of \eqref{cons}
and contains all solutions that remain bounded and sufficiently
close to a translate of $\bar U$ in forward time, where $p$ is the
(necessarily finite) number of unstable, i.e., positive real part,
eigenvalues of $L$.  
\et

Next, specializing a bit further, we add to (H0)--(H3) the additional
hypothesis that $\bar u$ be a {\it Lax-type shock}:
\medbreak

(H3) \quad The dimensions of the unstable subspace of $A_-$
and the stable subspace of $A_+$ sum to $n+1$.
\medbreak

We assume further the following {\it spectral genericity} conditions.
\medbreak

(D1) $L$ has no nonzero imaginary eigenvalues. 

(D2) The orbit $\bar U(\cdot)$ is a transversal connection of
the associated standing wave equation \eqref{ode}.

(D3) The associated inviscid shock $(U_-,U_+)$ is hyperbolically
stable, i.e.,
\be\label{liumajda}
\det(r_1^-,\dots, r_{P-1}^-, r_{P+1}^+, \dots , r_n^+, (U_+-U_-))\ne 0,
\ee
where $r_1^-, \dots r_{P-1}^-$ denote eigenvectors of $A_-$
associated with negative eigenvalues and
$r_{P+1}^+, \dots r_{n}^+$ denote eigenvectors of $A_+$
associated with positive eigenvalues.

\medbreak
\noindent As discussed in \cite{ZH,MaZ1}, (D2)--(D3) correspond in the
absence of a spectral gap to a generalized notion of simplicity of the 
embedded eigenvalue $\lambda=0$ of $L$.
Thus, (D1)--(D3) together correspond to the assumption that there are
no additional (usual or generalized) eigenvalues on the imaginary
axis other than the transational eigenvalue at $\lambda=0$;
that is, the shock is not in transition between different degrees of
stability, but has stability properties that are insensitive to 
small variations in parameters.

With these assumptions, we obtain our second main result
characterizing the stability properties of $\bar u$.
In the case $p=0$, this reduces to the nonlinear orbital stability
result established in \cite{MaZ1,MaZ2,MaZ3,Z2}.

\bt\label{t:mainstab}
Under (A1)--(A3),  (H0)--(H3), and (D1)--(D3), $\bar u$ is nonlinearly
orbitally stable under sufficiently small perturbations
in $L^1\cap H^3$ lying on the codimension $p$ center stable 
manifold $\cM_{cs}$ of $\bar u$ and its translates, 
where $p$ is the number of unstable eigenvalues of $L$,
in the sense that, for some $\alpha(\cdot)$, all $L^p$,
\ba\label{bounds}
|U(x, t)-\bar U(x-\alpha(t))|_{L^p}&\le
C(1+t)^{-\frac{1}{2}(1-\frac{1}{p})}
|U(x,0)-\bar U(x)|_{L^1\cap H^3},
\\
|U(x, t)-\bar U(x-\alpha(t))|_{H^3}&\le
C(1+t)^{-\frac{1}{4}} |U(x,0)-\bar U(x)|_{L^1\cap H^3},\\
\dot \alpha(t) &\le C(1+t)^{-\frac{1}{2}} |U(x,0)-\bar U(x)|_{L^1\cap H^3},
\\
\alpha(t) &\le C |U(x,0)-\bar U(x)|_{L^1\cap H^3}.
\ea
Moreover, it is orbitally unstable with respect to small $H^{1,2}$ perturbations
not lying in $\cM$, in the sense that the corresponding solution leaves
a fixed-radius neighborhood of the set of translates of $\bar U$ in
finite time.
\et

\br
\textup{
It is straightforward, combining the pointwise argument of \cite{RZ}
with the observations of \cite{Z6} in the strictly parabolic case,
to extend Theorem \ref{t:mainstab} to the case of nonclassical under- 
or overcompressive shocks with (D1)--(D3) suitably modified as described
in \cite{Z6}. 
This gives at the same time new information even in the Lax case,
including convergence of the phase $\alpha$ to a limit $\alpha(+\infty)$
at rate $(1+t)^{-1/2}$, with $|\dot \alpha|\le C(1+t)^{-1}$.
}
\er

\subsection{Discussion and open problems}\label{discussion}

For results of a similar nature in the context of the nonlinear 
Schr\"odinger equation, but obtained by rather different techniques,
we refer the reader to \cite{Sc}.
The method of \cite{Sc} consists of constructing the stable manifold
to a known center manifold using detailed linearized and nonlinear
estimates similar to those used to prove stability in the spectrally
stable case.
That is, the existence and stability arguments are carried out at the
same time.

This makes possible a somewhat finer analysis than the one
we carry out here (in keeping with the somewhat weaker stability
properties for the nonlinear Schr\"odinger equation, for which
the linearized operator about the wave possesses essential
spectrum on the whole of the imaginary axis lying outside a finite
ball), but also a somewhat less general one; in particular, it
is required that the linearized operator possess
no pure imaginary eigenvalues other than zero.
Note that we make no assumptions
on the pure imaginary spectrum of the linearized operator in our construction
of the center stable manifold, but only in the proof of conditional stability.

An interesting issue pointed out in \cite{TZ3}
is that the strong block structure assumptions (A1)--(A3) used
to obtain the key variational 
energy estimates on which our arguments are based hold for the
equations of gas dynamics or MHD written in Lagrangian coordinates,
but not in Eulerian coordinates, and indeed the energy estimates
themselves do not hold in that case; see Appendix A, \cite{TZ3}
for further discussion.
On the other hand, the 
%equations 
Lagrangian and Eulerian formulations are equivalent by a uniquely
specified (up to translation, which may be factored out) transformation,
and so the results 
%remain true also 
do hold
in Eulerian coordinates.
It is an interesting open question whether the results of this paper
hold for general symmetric hyperbolic--parabolic systems 
as defined in \cite{MaZ4,Z2,Z3}, which include among other examples
both Lagrangian and Eulerian formulations.

%%%%%%%%%%%%%%%%%%%

\section{Existence of Center Stable Manifold}\label{s:existence}

Defining the perturbation variable $V:=U-\bar U$, we
obtain after a brief computation the nonlinear perturbation
equations
\be\label{npert}
V_t - LV=N(V),
\ee
where $L$ as in \eqref{L} denotes the linearized operator about the wave and
$N=\cN(V)_x$ is a quadratic order residual, with
(recalling (A1))
\ba\label{N1}
\cN(V)&:= 
 \Big(B(\bar U+V)(\bar U+V)_{x}-B(\bar U)\bar U_{x} 
-B(\bar U)V_{x}- (dB(\bar U)V) \bar U_{x} \Big)  \\
&\quad
-\Big( F(\bar U+V) - F(\bar U) -dF(\bar U)V \Big)\\
&=
\bp O(|V|^2)\\ O(|V|^2 + |V||V_{2,x}|) \ep.
\ea
We seek to construct a Lipschitz $H^{1,2}$-local 
center stable manifold about the
equilibrium $v\equiv 0$, that is, a locally invariant Lipschitz manifold
tangent at $v\equiv 0$ to the center stable subspace $\Sigma_{cs}$,
that is quadratic-order tangent in the sense that
$|\Pi_{u}V|_{H^{1,2}}\le C |\Pi_{cs}V|_{H^{1,2}}^2$.

\subsection{Preliminary estimates}

Conditions (A2)--(A3) imply that
$
\Re\sigma \left( i\xi A - \xi^2 B\right)_\pm \leq -\frac{\theta_0 |\xi|^2}{1+
|\xi|^2}
$
for some $\theta_0>0,$ for all $\xi \in\BbbR$ \cite{KSh},
which in turn implies by standard spectral theory \cite{He}
that the essential spectrum of $L$ lies entirely in the neutrally stable
complex half-plane $\Re \lambda\le 0$, and the unstable
half-plane $\Re \lambda>0$ either contains finitely many eigenvalues
of finite multiplicity, or else consists entirely of eigenvalues,
as is easily shown (by energy estimates, for example) to be impossible.
Thus, there is a well-defined unstable subspace $\Sigma_u$ 
consisting of the direct
sum of the finitely many generalized eigenfunctions in the positive half-plane,
and an associated unstable eigenprojection $\Pi_u$.
Likewise, there is a complementary center--stable subspace $\Sigma_{cs}$
determined by the range of the complementary projection 
$\Pi_{cs}:=\Id- \Pi_u$.

\begin{proposition}[\cite{MaZ3,Z2,Z3}]\label{linest}
Under assumptions (A1)--(A3), (H0)--(H2),
$L$ generates a $C^0$ semigroup $e^{Lt}$ satisfying
 \begin{equation} \label{bound-pde} 
\begin{aligned}
| e^{t L}\Pi_{cs}  |_{L^2\to L^2}  &\leq  C_\o   e^{\o t}, \\
| e^{-t L}\Pi_u  |_{L^2\to H^{1,2}}  &\leq  C_\o e^{-\b t},\\
\end{aligned}
 \end{equation}
 for some $\b > 0,$ and for all $\o > 0,$ for all $t\ge 0$.
\end{proposition}

\begin{proof}
Straightforward energy estimates yield resolvent bounds 
verifying the $C^0$-semigroup property; see \cite{MaZ3,Z2} 
and especially \cite{Z3}, Prop. 3.6.
The same estimates yield
$$
|(\lambda-L)^{-1}|_{L^2\to L^2}\le C(\eta)
$$
for $\lambda \ge \eta$ and $|\lambda|$ sufficiently large, 
$\eta>0$ arbitrary, whereupon
\eqref{bound-pde}(i) follows by Pr\"uss' theorem
restricted to the center--stable subspace \cite{Pr,ProK,KS};
see \cite{Z3}, Appendix A for further discussion,
in particular Rmk. 6.23. 
Bound \eqref{bound-pde}(ii) follows
by equivalence of norms for finite-dimensional spaces and standard
finite-dimensional ODE estimates.
\end{proof}

Introducing a $C^\infty$ cutoff function
$$
\rho(x) =\begin{cases}
1  & | x | \leq 1, \\ 
0 &  | x | \geq 2,\end{cases}
$$
and recalling mixed-norm definition \eqref{mixed}, let 
$$
 N^\delta (V) := 
\rho\Big( \frac{ | V |_{H^{1,2}}}{\delta}\Big) N(V).
$$

 \begin{lemma}\label{trunc} 
Assuming (A1)--(A3), (H0)--(H2), the map $N^\delta: H^{1,2} 
\to L^2 $ is $C^{k+1}$ and its Lipschitz norm with respect to $V$ 
is $O(\delta)$ as $\delta\to 0.$
Moreover, 
\be\label{Nquad}
|N^\delta(V)|_{L^2}\le C|V|_{H^{1,2}}^2.
\ee
 \end{lemma}

 \begin{proof}
 The norm in $H^{1,2}$ is a quadratic form, hence the map
 $$ V \in H^{1,2} \mapsto \rho\Big( \frac{ | V |_{H^2}}{\delta}\Big) 
\in \R_+$$
 is smooth, and $N^\delta$ is as regular as $N.$ 
By Moser's inequality, 
\ba\nonumber
| N(V)|_{L^2} &\le C(|V|_{H^1}|V|_{L^\infty}
+ |\partial_x V_2|_{H^1}|V|_{L^\infty}|
+ |\partial_x V_2|_{L^\infty}|V|_{H^1}|)
\le C|V|_{H^{1,2}}^2
\ea
and, similarly,
$| N(V^1)-N(V^2))|_{L^2} \le 
C|V^1-V^2|_{H^{1,2}}| \sup_{j=1,2} C|V^j|_{H^{1,2}}|$
and thus $|dN|_{H^{1,2}\to L^2}\le C|V|_{H^{1,2}}$
so long as $|V|_{H^{1,2}}$ remains bounded, in particular
for $|V|_{H^{1,2}}\le \delta$.
Thus, 
$ |N^\delta(V)|_{L^2}\le |N(V)|_{L^2}\le C|V|_{H^{1,2}}^2 $
for $|V|_{H^{1,2}}\le \delta$, while 
$ N^\delta(V)=0$ for $|V|_{H^{1,2}}\ge \delta$, 
verifying \eqref{Nquad}.
The Lipschitz bound follows, likewise, by
\ba\nonumber
| N^\delta(V^1) - N^\delta(V^2)|_{L^2} 
& \leq  \Big| \rho\Big( \frac{ | V^1 |_{H^{1,2}}}{\delta}\Big) -   
 \rho\Big( \frac{ | V^2 |_{H^{1,2}}}{\delta}\Big) \Big|_{L^\infty} 
| N(V^1)|_{L^2} \\ 
 & \quad  +  
\Big| \rho\Big( \frac{ | V^2 |_{H^{1,2}}}{\delta}\Big) \Big|_{L^\infty}  
| N(V^1) - N(V^2)|_{L^2} \\
 & \leq  3 | V^1 - V^2|_{H^{1,2}} \\
&\times
\Big(
\sup_{|V|_{H^{1,2}} < \delta} \frac{| N(V)|_{L^2}}{\delta}
+   
\sup_{|V|_{H^{1,2}} < \delta} | dN(V)|_{H^{1,2}\to L^2}
\Big) .
\ea
 \end{proof}

\begin{corollary}\label{integrand}
Under assumptions (A1)--(A3), (H0)--(H2),
 \begin{equation} \label{bound-integrand} 
\begin{aligned}
| e^{t L}\Pi_{cs} N^\delta |_{H^{1,2}\to L^2}  &\leq  
C_\o \delta e^{\o t}, \\
| e^{-t L}\Pi_u N^\delta |_{H^{1,2}\to H^{1,2}}  &\leq  C_\o \delta e^{-\b t}, \\
\end{aligned}
 \end{equation}
for some $\b > 0,$ and for all $\o > 0,$ for all $t\ge 0$,
with Lipschitz bounds
 \begin{equation} \label{bound-integrand-lip} 
\begin{aligned}
| e^{t L}\Pi_{cs} dN^\delta |_{H^{1,2}\to L^2}  &\leq  
C_\o \delta  e^{\o t}, \\
| e^{-t L}\Pi_u dN^\delta |_{H^{1,2}\to H^{1,2}}  &\leq  C_\o  \delta e^{-\b t}. \\
\end{aligned}
 \end{equation}
\end{corollary}

\subsection{Fixed-point iteration scheme}\label{it}

Applying projections $\Pi_j$, $j=cs,u$ to the truncated equation
\be\label{qpert}
V_t-LV=N^\delta(V),
\ee
we obtain using the variation of constants formula equations
$$
\Pi_j V(t)=
e^{L(t-t_{0,j})}\Pi_j V(t_{0,j}) + \int_{t_{0,j}}^t e^{L(t-s)}\Pi_j  N^\delta(V(s))\,ds,
$$
$j=cs,u$, so long as the solution $V$ exists,
with $t_{0,j}$ arbitrary.
Assuming growth of at most $|V(t)|_{H^2}\le Ce^{\tilde \theta t}$ in positive 
time, we find for $j=u$ using bounds \eqref{bound-pde}(ii)
and \eqref{bound-integrand-lip}(ii) 
that, as $t_{0,u}\to +\infty$, the first term
$e^{L(t-t_{0,u})}\Pi_u V(t_{0,u})$ converges to zero while
the second, integral term
converges to $\int_{t}^{+\infty} e^{L(t-s)}\Pi_u  N^\delta(V(s))\,ds$,
so that, denoting $w:=\Pi_{cs}V$, $z:=\Pi_u V$, we have
\be\label{vu}
z(t)=\cT(z,w)(t):=  -\int_t^{+\infty} e^{L(t-s)}\Pi_u  N^\delta( (w+z)(s))\,ds.
\ee

Likewise, choosing $t_{0,cs}=0$, we have
\be\label{vcs}
w(t)=
e^{Lt}\Pi_{cs} w_0 + \int_{0}^t e^{L(t-s)}\Pi_{cs}  N^\delta( (w+z)(s))\,ds,
\ee
$w_0:=\Pi_{cs}V(0)$.
On the other hand, 
we find from the original differential equation projected onto
the center stable component
%, after some rearrangement, 
that $w$ satisfies the Cauchy problem
\ba\label{diffvcs}
w_t- Lw&=
%THIS TOO..
%\Pi_{cs}Lz 
%NO! this is zero!
%- \Pi_u Lw +
\Pi_{cs}N^\delta(w+z)\\
\ea
with initial data $ w_0=\Pi_{cs}V(0)$ given at $t=0$.
% where
%\be\label{M}
%M(V):= (dB(\bar u)V) \bar u_x -h_u(\bar u, \bar u_x)V 
%- h_{u_x}(\bar u, \bar u_x)V_x.
%\ee
We shall use these two representations together to obtain optimal
estimates, the first for decay, through standard linear semigroup
estimates, and the second for regularity,
through the nonlinear damping estimates \eqref{Ebounds1}--\eqref{Ebounds2}
and \eqref{LEbounds1}--\eqref{LEbounds2} below.

Viewing \eqref{vcs}, or alternatively \eqref{diffvcs},
as determining $w=\cW(z,w_0)$ as a function of $z$, we seek $z$ as
a solution of the fixed-point equation
\be\label{fixedpt}
z=\tilde \cT(z,w_0):=\cT(z, \cW(z,w_0)).
\ee
As compared to the standard ODE construction of, e.g., \cite{B,VI,TZ1,Z5},
in which \eqref{vu}--\eqref{vcs} together are considered as a fixed-point
equation for the joint variable $(w,z)$, this amounts to treating
$w$ implicitly.
This is a standard device in situations of limited regularity;
see, e.g., \cite{CP,GMWZ,RZ}.

It remains to show, first, that $\cW$, hence $\tilde\cT$, is well-defined 
on a space of slowly-exponentially-growing functions and,
second, that $\tilde\cT$ is contractive on that space, 
determining a $C^k$ solution
$z=z(w_0)$ similarly as in the usual ODE construction.
We carry out these steps in the following subsections.

\subsection{Nonlinear energy estimates}\label{nde}

Define now the negatively-weighted sup norm 
$$
\|f\|_{-\eta}:=\sup_{t\ge 0}e^{-\eta t}|f(t)|_{H^{1,2}},
$$
noting that $ |f(t)|_{H^2}\le e^{\tilde \theta t} \|f\|_{-\tilde \theta}$ 
for all $t\ge 0$, and denote by $\cB_{-\eta}$ the Banach space
of functions bounded in $\|\cdot\|_{-\eta}$ norm.
Define also the auxiliary norm 
$$
\|f\|_{L^2_{-\eta}}:=\sup_{t\ge 0}e^{-\eta t}|f(t)|_{L^2}.
$$

\bl[\cite{Z5}]\label{projlem1}
Under assumptions (A1)--(A3), (H0)--(H2), 
for all $1\le p\le \infty$, $0\le r\le 4$, 
\ba
|\Pi_{u}|_{L^p\to W^{r,p}} \, , \;
|\Pi_{cs}|_{W{r,p}\to W^{r,p}}&\le C.
\ea
\el

\begin{proof} 
Recalling that $L$ has at most finitely many unstable eigenvalues,
we find that $\Pi_u$ may be expressed as
$$
\Pi_u f= \sum_{j=1}^p \phi_j(x) \langle \tilde \phi_j, f\rangle,
$$
where $\phi_j$, $j=1, \dots p$ are generalized right eigenfunctions of
$L$ associated with unstable eigenvalues $\lambda_j$, 
satisfying the generalized eigenvalue equation $(L-\lambda_j)^{r_j}\phi_j=0$,
$r_j\ge 1$, and $\tilde \phi_j$
are generalized left eigenfunctions.
Noting that $\phi_j$, $\tilde \phi_j$ and derivatives decay exponentially
by standard theory \cite{He,ZH,MaZ1}, and estimating 
$$
|\partial_x^j\Pi_u f|_{L^p}=|\sum_j \partial_x^j\phi_j \langle \tilde \phi_j f
\rangle|_{L^p}\le
\sum_j |\partial_x^j\phi_j|_{L^p} |\tilde \phi_j|_{L^q} |f|_{L^p}
\le C|f|_{L^p}
$$
for $1/p+1/q=1$, we obtain the claimed
bounds on $\Pi_u$, from which the bounds on
$\Pi_{cs}={\rm Id}-\Pi_u$ follow immediately.
\end{proof}

\bpr\label{localdamp}
Assuming (A1)--(A3), (H0)--(H2),
denote $A^0:=A^0(\bar u(x))$ and take
without loss of generality $\sigma(A_{11}< 0$.
Then, there exist a skew-symmetric matrix $K(x)$ and
constants $C>0$, $C_*>0$ such that, defining the
scalar weight $\alpha$ by 
$\alpha(0)=1$ and $ \alpha_x= C_*|\bar U_x|\alpha$,
the quadratic form
$$
\cE(w):=
\langle (C^2 A^0 + CK\partial_x - A^0 \partial_x^2)\alpha w, w\rangle
$$
is equivalent to $|w|_{H^1}^2$ and, for $w_t-Lw=\bp 0\\r\ep$,
\be\label{Erel}
\partial_t \cE(w)\le -\theta |w|_{H^{1,2}}^2 + C(|f|_{L^2} + |w|_{L^2}^2).
\ee
\epr

\begin{proof}
This follows by a linear version of the argument for Proposition 4.15,
\cite{Z3}, somewhat simplified by the stronger block
structure assumptions (A1)--(A3) made here.
We first observe that, by symmetry of $(A^0A)_{11}$ and
$(A^0A)_\pm$,
$$
A^0A(x)= \tilde A(x)+ 
\bp 0 & O(|\bar U_x|)\\
O(|\bar U_x|)&
O(|\bar U_x|)\ep
$$
where $\tilde A$ is the symmetric part of $A^0A$.
Next, we recall from \cite{Kaw,KSh} that (A2)--(A3) imply
existence of skew-symmetric $K_\pm$ such that $\Re (KA + B)_\pm >0$.
Defining $K(x)$ as a smooth interpolant between $K_\pm$ defined
by arc length along the path of $\bar U$ between $U_\pm$,
we obtain the result by a straighforward (if somewhat lengthy)
computation for $C>0$ and $C_*>>C$ $C>0$ sufficiently large,
where ``good'' terms 
$$
-\langle \partial_x^2 w, \alpha A^0B \partial_x^2 w>
\le -\theta \langle \partial_x^2 w_2, \alpha \partial_x^2 w_2|,
$$
$$
-C^2\langle \partial_x w, \alpha A^0B \partial_x w>
\le -C^2\theta \langle \partial_x w_2, \alpha \partial_x w_2|,
$$
$$
\begin{aligned}
-C\langle \partial_x w, \Re (KA + A^0B)\alpha \partial_x w\rangle
&\le -C\theta \langle \partial_x w, \alpha \partial_x w\rangle\\
&+
C_2\langle \partial_x w_1, |\bar U_x| \alpha \partial_x w_1\rangle
+C_2 \langle \partial_x w_2, \alpha \partial_x w_2\rangle,
\end{aligned}
$$
and
$$
\begin{aligned}
-\langle \partial_x w_1, (\alpha A^0A)_{11}\partial_x^2 w_1\rangle
&=
\langle \partial_x w_1, \partial_x (\alpha A^0A)_{11}\cdot \partial_x w_1\rangle
\\
& \le -\theta C_* \langle \partial_x w_1, |\bar U_x|\alpha \partial_x w_1\rangle
\end{aligned}
$$
resulting by rearrangement/integration by parts of the various components
of $\partial_t \cE(w)$,
when summed together, after absorbing remaining terms, give a contribution of
$
-\theta |w|_{H^{1,2}}^2 + C(|w|_{L^2}^2  + |f|_{L^2}^2),
$
verifying the result.
See \cite{MaZ4,Z2,Z3} for further details.
\end{proof}

\bc\label{damp1}
Under assumptions (A1)--(A3), (H0)--(H2),
for $\|z\|_{L^2_{-\eta}}$ bounded and $\delta$ sufficiently small, 
the solution $w$ of \eqref{diffvcs} exists for all
$t\ge 0$ and, for any constant $\theta > 0$ and some $C=C(\theta)$, 
\begin{equation}\label{Ebounds1}
\begin{aligned}
 \int_0^t e^{-\theta s} & |w|_{H^{1,2}}^2(s) \,ds 
+e^{-\theta t}|w_{t}|_{H^1}^2
&\le
C|w_{0}|_{H^1}^2
+ C \int_0^t e^{-\theta s} (|w|_{L^2}^2 + |z|_{L^2}^2) (s)\,ds ;
\end{aligned}
\end{equation}
likewise, for $\theta>2\eta$,
\begin{equation}\label{Ebounds2}
\begin{aligned}
 \int_t^{+\infty} e^{\theta (t-s) }  |w|_{H^{1,2}}^2(s) \,ds &\le
C|w|_{H^1}^2(t)
+ C \int_t^{+\infty} e^{\theta (t- s)} (|w|_{L^2}^2 + |z|_{L^2}^2) (s)\,ds \\
& \le
Ce^{\theta t}|w_{0}|_{H^1}^2 
+C e^{2\eta t}(\|w\|_{L^2_{-\eta}}^2 + \|z\|_{L^2_{-\eta}}^2 ).
\end{aligned}
\end{equation}
\ec

\begin{proof}
Observing that 
\be\label{nlbd}
|\Pi_{cs}N^\delta(w+z)|_{L^2}\le
\bp 0\\
C\delta(|w|_{H^{1,2}}+ |z|_{H^{1,2}})\ep,
\ee
we obtain by Proposition \ref{localdamp} the inequality
$$
\begin{aligned}
\partial_t \cE(w)
& \leq -{\tilde \theta}|w|_{H^{1,2}}
+ C\delta \big(|w|_{H^{1,2}}^2 + |z|_{H^{1,2}}^2\big)+
 C \big(|w|_{L^2}^2 + |z|_{L^2}^2\big)\\
\end{aligned}
$$
for some $\tilde \theta>0$,
which, for $\delta$ sufficiently small, gives
$$
\begin{aligned}
\partial_t \cE(w)
& \leq -{\tilde \theta}|w|_{H^{1,2}}
 + C_2 \big(|w|_{L^2}^2 + |z|_{L^2}^2\big)\\
\end{aligned}
$$
by $|z|_{H^{1,2}}|\le C|z|_{L^2}$
(equivalence of finite-dimensional norms).
From this inequality, global $H^1$ existence follows immediately
and \eqref{Ebounds1} and the first line of \eqref{Ebounds2} follow
readily by Gronwall's inequality together with
equivalence of $\cE$ and $|w|_{H^1}^2$.
For further discussion, see
\cite{Z6}, proof of Proposition 2.6.

Bounding $|w(t)|_{H^1}^2$ in the second line of \eqref{Ebounds2}
using \eqref{Ebounds1} now yields
\begin{equation}\label{Ebounds}
\begin{aligned}
 \int_t^{+\infty} e^{\theta (t-s) }  |w|_{H^{1,2}}^2(s) \,ds &\le
Ce^{\theta t}\Big(|w_{0}|_{H^1}^2 
+  \int_0^{+\infty} e^{-\theta  s} (|w|_{L^2}^2 + |z|_{L^2}^2) (s)\,ds
\Big),
\end{aligned}
\end{equation}
whereupon the final line of \eqref{Ebounds2} then follows by
$$
\begin{aligned}
 \int_0^{+\infty} e^{\theta (t- s)} 
(|\dot w|_{L^2}^2 + |\dot z|_{L^2}^2) (s)\,ds 
&\le 
\Big( \e^{\theta t}\int_t^{+\infty} e^{(2\eta-\theta)  s} \, ds\Big)
(\|\dot w\|_{L^2_{-\eta}}^2 + \|\dot z\|_{L^2_{-\eta}}^2)
\\
&\le
C e^{2\eta t}(\|\dot w\|_{L^2_{-\eta}} + \|\dot z\|_{L^2_{-\eta}} ).
\end{aligned}
$$
\end{proof}

\bc\label{Lip1}
Under (A1)--(A3), (H0)--(H2),
for $\|z_1\|_{L^2_{-\eta}}$, $\|z_2\|_{L^2_{-\eta}}$ 
 bounded and $\delta$ sufficiently small, 
solutions $w_1,z_1$ and $w_2,z_2$ of \eqref{diffvcs} exist for all
$t\ge 0$ and, for any constant $\theta > 0$ and some $C=C(\theta)$, 
\begin{equation}\label{LEbounds1}
\begin{aligned}
 \int_0^t e^{-\theta s} & |w_1-w_2|_{H^{1,2}}^2(s) \,ds \le
C|w_{0,1}-w_{0,2}|_{H^1}^2\\
&\quad +
C \int_0^t e^{-\theta s} (|w_1-w_2|_{L^2}^2 + |z_1-z_2|_{L^2}^2) (s)\,ds ;
\end{aligned}
\end{equation}
likewise, for $\theta>2\eta$,
\begin{equation}\label{LEbounds2}
\begin{aligned}
 \int_t^{+\infty} e^{\theta (t-s) } & |w_1-w_2|_{H^{1,2}}^2(s) \,ds \le
C|w_1-w_2|_{H^1}^2(t)\\
&\quad +
C \int_t^{+\infty} e^{\theta (t- s)} 
(|w_1-w_2|_{L^2}^2 + |z_1-z_2|_{L^2}^2) (s)\,ds \\
&
\qquad \qquad \qquad
\quad
\le
Ce^{\theta t}|w_{0,1}-w_{0,2}|_{H^1}^2\\
&\quad 
+C e^{2\eta t}(\|w_1-w_2\|_{L^2_{-\eta}}^2 + \|z_1-z_2\|_{L^2_{-\eta}}^2 ).
\end{aligned}
\end{equation}
\ec

\begin{proof}
Subtracting the equations for $w_1,z_1$ and $w_2,z_2$, we obtain,
denoting $\dot w:=w_1-w_2$, $\dot z:=z_1-z_2$, the equation
\ba\label{vardiffvcs}
\dot w_t- L\dot w=
\Pi_{cs}\big( N^\delta(w_1+z_1) -N^\delta(w_2+z_2) \big)\\
\ea
with initial data $ \dot w_0=w_{0,1}-w_{0,2}$ at $t=0$,
whence the result follows by the Lipshitz bound $\delta$
on $N^\delta:H^{1,2}\to L^2$ and the same argument used
in the proof of Corollary \ref{damp1}.
\end{proof}

\bc\label{Texist}
Under assumptions (A1)--(A3), (H0)--(H2),
for $3\o<\eta<\beta$ and $\delta>0$ and $w_0\in H^2$ sufficiently small,
for each $z\in \cB_{-\eta}$,
there exists a unique solution $w=:\cW(z,w_0)\in \cB_{-\eta}$ 
of \eqref{vcs}, \eqref{diffvcs}, with
%\be\label{comp}
%\|w\|_{-\eta} \le C(|w_0|_{H^2}+\delta\|z\|_{-\eta})
%\ee
%and 
\be\label{comp}
\|w\|_{L^2_{-\eta}} \le C(|w_0|_{H^{1,2}}+\delta\|z\|_{-\eta})
\ee
and 
\be\label{Lipcomp}
\|\cW(z_1, w_{0,1})-\cW(z_2, w_{0,2})\|_{L^2_{-\eta}}
\le C\delta \|z_1-z_2\|_{-\eta}
+ C\|w_{0,1}-w_{0,2}\|_{H^{1,2}}.
\ee
\ec

\begin{proof}
We have already shown existence, while uniqueness follows from 
\eqref{Lipcomp}.  Thus, we need only verify the estimates.

Consider a pair of data $z_1, w_{0,1}$ and
$z_2, w_{0,2}$, and compare the resulting solutions,
denoting 
$
(\dot z, \dot w , \dot w_0):=(z_1-z_2,w_1-w_2,w_{0,1}-w_{0,2}).
$
Using the integral representation \eqref{vcs}, and applying
\eqref{bound-pde}(i), \eqref{bound-integrand-lip}(i),
and the definition of $\|\cdot\|_{-\eta}$, we obtain
for all $t\ge 0$ 
\ba\label{tempz}
|\dot w(t)|_{L^2}&\le C e^{\o |t|}|\dot w_{0}|_{L^2}
+ C\delta \int_{0}^t e^{  \o |t-s|} 
 (|\dot w\dot |_{H^2 }+ |\dot z|_{H^2 })(s) \, ds.
\ea

Estimating 
$$
 C\delta \int_{0}^t e^{  \o (t-s)} |\dot z|_{H^2 }(s) \, ds
\le
 C\delta  \|\dot z\|_{-\eta}\int_{0}^t
e^{  \o (t-s)} e^{\eta s} \, ds
\le
 C\delta  \|\dot z\|_{-\eta} e^{\eta t}
$$
and, by \eqref{LEbounds1} with $\theta=3\o$ together with the
Cauchy--Schwarz inequality, 
$$ 
\begin{aligned}
C\delta \int_{0}^t e^{  \o (t-s)} 
 |\dot w\dot |_{H^2 }(s) \, ds &\le 
C\delta \Big(\int_{0}^t e^{  -\o (t-s)} \Big)^{1/2}
\Big(\int_{0}^t e^{3  \o (t-s)} 
 |\dot w\dot |_{H^2 }^2(s) \, ds \Big)^{1/2}\\
 &\le 
C\delta \Big(e^{3\o t}|w_0|_{H^1}^2+
\int_{0}^t e^{ 3 \o (t-s)} 
 (|\dot w\dot |_{L^2 }^2+ |\dot z|_{L^2}^2)(s) \, ds \Big)^{1/2},
\end{aligned}
$$
we obtain, substituting in \eqref{tempz}, 
\ba\label{tempz2}
|\dot w(t)|_{L^2}^2&\le C e^{6\o |t|}|\dot w_{0}|_{H^1}^2
+ C\delta^2 \Big(  \|\dot z\|_{-\eta}^2 e^{2\eta t}
+  \int_{0}^t e^{  3\o (t-s)} 
 (|\dot w\dot |_{L^2 }^2+ |\dot z\dot |_{L^2 }^2) (s) \, ds \Big)\\
&\le
 e^{2\eta t} \Big( C|\dot w_{0}|_{H^1}^2
+ C\delta^2 ( \|\dot w\|_{L^2_{-\eta}}^2 + \|\dot z\|_{-\eta}^2)\Big).
\ea

This yields 
\ba\nonumber
\|\dot w(t)\|_{L^2_{-\eta}}&\le C |\dot w_{0}|_{-\eta}
+ C\delta ( \|\dot w\|_{L^2_{-\eta}} + \|\dot z\|_{-\eta}),
\ea
from which \eqref{Lipcomp} follows by smallness of $\delta$.
The bound \eqref{comp} follows similarly.
\end{proof}

%%%%%%%%%%%%%%%%%%%%%%%%%%%%%%%%%%%%%%%%%%%%%%%%%%%%%%%%%%%%%%%%%%%%%
\subsection{Basic existence result}\label{pdecsproof}

\begin{proof}[Proof of Proposition \ref{t:maincs} without
translational independence] 
{\it (i) (Contraction mapping argument)}
We find using \eqref{Lipcomp}, \eqref{bound-integrand-lip}(ii), 
and Lemma \ref{trunc} that
\ba\label{tempz5}
\|\tilde \cT(z_1,w_{0,1})&-\tilde \cT(z_2,w_{0,2})\|_{-\eta}\\
&
\le \sup_t C\delta  e^{-\eta t}\int_t^{+\infty} e^{\beta(t-s)} (|\dot w|_{H^2}+|\dot z|_{H^2})(s) \, ds
\\ &
\le \sup_t C_1 \delta
\Big( \|\dot z\|_{-\eta }
+
e^{-\eta t} \int_t^{+\infty} e^{\beta(t-s)} |\dot w|_{H^2}(s)  ds \Big).
\\
\ea
Using the Cauchy--Schwarz inequality and
\eqref{LEbounds2} with $\theta=\beta$ to estimate
\ba\nonumber
\int_t^{+\infty} e^{\beta(t-s)} |\dot w|_{H^2}(s)  ds &\le
\Big(  \int_t^{+\infty} e^{\beta(t-s)} ds \Big)^{1/2}
%\\
%&\quad \times
\Big(  \int_t^{+\infty} e^{\beta(t-s)} |\dot w|_{H^2}^2(s)  ds \Big)^{1/2}
\\
&\le
C_3\Big(
|\dot w|_{H^1}^2(t) +
 \int_0^t e^{\beta (t-s)} (|\dot w|_{L^2}^2 + |\dot z|_{L^2}^2) (s)ds
\Big)^{1/2}\\
&\le
C_3e^{\eta t}\Big(|\dot w_{0}|_{H^1}
+\|\dot w\|_{L^2_{-\eta}} + \|\dot z\|_{L^2_{-\eta}} \Big)
\\
\ea
and applying \eqref{Lipcomp}, we obtain 
\ba\label{finalLip}
\|\tilde \cT(z_1,w_{0,1})-\tilde \cT(z_2,w_{0,2})\|_{-\eta}&\le
C\delta ( \|\dot w_0\|_{H^{1,2}} + \|\dot z\|_{L^2_{-\eta}}).
\ea
Parallel estimates yield 
\ba\nonumber
\|\tilde T(z,w_0)(t)\|_{-\eta}
&\le
C\delta 
(|w_0|_{H^{1,2}}+\delta\|z\|_{-\eta}),
\ea
so that,
taking $\delta $ and $|w_0|_{H^{1,2}}$ sufficiently small, that
$\tilde \cT(\cdot, w_0)$ maps the ball $B(0,r)\subset \cB_{-\eta}$
to itself, for $r>0$ arbitrarily small but fixed.

This yields at once contractivity on $B(0,r)$,
hence existence of a unique fixed point $z=\cZ(w_0)$, and
Lipshitz continuity of $\cZ$ from $\Sigma_{cs}$ to $\cB_{-\eta}$,
%with Lipschitz constant $\delta$ arbitrarily small, 
by the Banach Fixed-Point Theorem
with Lipschitz dependence on parameter $w_0$.

{\it (ii) (Existence of a Lipschitz invariant manifold)}
Defining
\ba\label{Phidefc}
\Phi(w_{0})&:= \cZ(w_0)|_{t=0} =
 -\int_{0}^{+\infty}e^{-Ls}\Pi_{u}  N^\delta (v(s))\,ds,
\ea
we obtain a Lipschitz function from $\Sigma_{cs}\to \Sigma_{u}$,
whose graph over $B(0,r)$ is the invariant manifold of solutions 
of \eqref{qpert} growing at exponential rate $|v(t)|\le Ce^{\eta t}$ 
in forward time.  From the latter characterization, we obtain evidently
invariance in forward time.
Since the truncated equations \eqref{qpert} agree with the original
PDE so long as solutions remain small in $H^{1,2}$, this gives local invariance
with respect to \eqref{cons} as well.
By uniqueness of fixed point solutions, we have $\cZ(0)=0$ and
thus $\Phi_{cs}(0)=0$, so that the invariant manifold passes through
the origin.
Likewise, any bounded, sufficiently small solution of \eqref{npert} 
in $H^{1,2}$ is a bounded small solution of \eqref{qpert} as well, so
by uniqueness is contained in the center stable manifold.

{\it (iii) (Quadratic-order tangency)}
By \eqref{Phidefc}, \eqref{bound-pde}, \eqref{Nquad}, 
\eqref{Ebounds2}, amd \eqref{comp},
\ba\nonumber
|\Phi(w_{cs})|_{H^{1,2}} &=
 \Big|\int_{0}^{+\infty}e^{-Ls}\Pi_{u}  N^\delta (w+z)(s)\,ds
\Big|_{H^{1,2}}\\
&\le C\int_0^{+\infty} e^{-\beta s}(|\cW|_{H^{1,2}}^2 +|\cZ|_{H^{1,2}}^2)(s)ds\\
%&\le C_1(\|\cW\|_{L^2_{-\eta}}^2 +\|\cZ\|_{-\eta}^2)
&
\le C_2(|w_{cs}|_{H^{1,2}}^2 +\|\cZ\|_{-\eta}^2).\\
\ea
By $\cZ(0)=0$ and Lipshitz continuity of $\cZ$, we have
$\|\cZ\|_{-\eta}\le C|w_{cs}|_{H^{1,2}}$, whence
\be\label{tanbd}
|\Phi(w_{cs})|_{H^{1,2}}\le C_3 |w_{cs}|_{H^{1,2}}^2,
\ee
verifying quadratic-order tangency at the origin.
\end{proof}

\subsection{Translation-invariance}\label{s:trans}

We conclude by indicating briefly how to recover translation-invariance
of the center stable manifold, following \cite{Z5,TZ1}.
Differentiating with respect to $x$ the traveling-wave ODE, 
we recover the standard fact that $ \phi := \bar U_x $
is an $L^2$ zero eigenfunction of $L$, by the decay of $\bar U_x$
noted in Remark \ref{profdecay}.

Define orthogonal projections
\begin{equation}
\label{proj}
\Pi_2 :=
\frac{ \phi \, \langle \phi, \cdot\rangle}{|\phi|_{L^2}^2}, 
\qquad \Pi_1 := \Id - \Pi_2,
\end{equation}
onto the range of right zero-eigenfunction $\phi:= \bar U_x$
of $L$ and its orthogonal complement $\phi^\perp$ in $L^2$, where 
$\langle \cdot, \cdot \rangle$ denotes standard $L^2$ inner product.

\begin{lemma}\label{Pi}
Under the assumed regularity $h\in C^{k+1}$, $k\ge 2$,
$\Pi_j$, $j=1,2$ are bounded as operators from $H^s$ to itself
for $0\le s \le k+2$.
\end{lemma}

\begin{proof}
Immediate, by the assumed decay of $\phi=\bar u_x$ and derivatives.
\end{proof}

\begin{proof}[Proof of Proposition \ref{t:maincs}, with
translational independence] 
Introducing the shifted perturbation variable
\be\label{shift}
V(x,t):= U(x+\alpha(t),t)-\bar U(x)
\ee
we obtain the modified nonlinear perturbation equation
\begin{equation}
\label{epspert}
\d_t V= LV + \cN(V)_x -\d_t \alpha (\phi+ \partial_x V),
\end{equation}
where $ L := \frac{\d {\cal F}}{\d U}( \bar U)$ and
$\cN$ as in \eqref{N1} is a quadratic-order Taylor remainder.

Choosing $\d_t\alpha$ so as to cancel $\Pi_2$ of the righthand
side of \eqref{epspert}, we obtain the {\it reduced equations}
\be\label{redv}
\d_t V= \Pi_1( LV + N(V))
\ee
and
\be\label{alphaeq}
\d_t \alpha =
\frac{\pi_2(LV + N(V))}
{1+ \pi_2 (\d_x V)}
\ee
$V\in \phi^\perp$, 
$\pi_2 V:= \langle \tilde \phi, V\rangle |\phi|_{L^2}^2$,
of the same regularity as the original equations.

$V(0)\in \phi^\perp$, or
$$
\langle \phi, U_0(x+\alpha)-\bar U(x)\rangle=0.
$$
Assuming that $U_0$ lies in a sufficiently small tube about
the set of translates of $\bar U$, or $U_0(x)=(\bar U+W)(x-\beta)$
with $|W|_{H^{1,2}}$ sufficiently small, this can be done in a unique
way such that $\tilde \alpha:=\alpha-\beta$ is small, as determined
implicitly by 
$$
0=\cG(W,\tilde \alpha):=
\langle \phi, \bar U(\cdot+\tilde \alpha)-\bar U(\cdot)\rangle
+\langle \phi, W(\cdot+\alpha) \rangle,
$$
an application of the Implicit Function Theorem noting that
$$
\d_{\tilde \alpha}\cG(0,0)=
 \langle \phi, \partial_x \bar U(\cdot)\rangle=|\phi|_{L^2}^2\ne0.
$$
With this choice, translation invariance under our construction is clear,
with translation corresponding to a constant shift in $\alpha$
that is preserved for all time.

Clearly, \eqref{alphaeq} is well-defined so long as 
$|\d_x V|_{L^2}\le C|V|_{H^{1,2}}$
remains small, hence we may solve the $v$ equation independently
of $\alpha$, determining $\alpha$-behavior afterward to determine
the full solution 
$$
U(x,t)= \bar U(x-\alpha(t))+V(x-\alpha(t),t).
$$
Moreover, it is easily seen (by the block-triangular structure
of $L$ with respect to this decomposition) that
the linear part $\Pi_1L=\Pi_1 L\Pi_1$ of the $v$-equation possesses
all spectrum of $L$ apart from the zero eigenvalue associated with 
eigenfunction $\phi$.
Thus, we have effectively projected out this zero-eigenfunction, and
with it the group symmetry of translation.

We may therefore construct the center stable manifold for the reduced
equation \eqref{redv}, automatically obtaining translation-invariance
when we extend to the full evolution using \eqref{alphaeq}.
See \cite{TZ1,Z5} for further details.
\end{proof}

%%%%%%%%%%%%%%%%%%%

\section{Conditional stability analysis}\label{s:cond}

Define similarly as in Section \ref{s:trans} the perturbation variable
\be\label{pert2}
V(x,t):=U(x+\alpha(t),t)-\bar U(x)
\ee
for $U$ a solution of \eqref{cons}, where $\alpha$ is to be specified later.
Subtracting the equations for $U(x+\alpha(t), t)$ and $\bar U(x)$,
we obtain the nonlinear perturbation equation
\be\label{nlpert}
V_t-LV= \cN(V)_x - \dot \alpha(\bar U_x + V_x),
\ee
where $L:= -A\partial_x +\partial_x B \partial_x$ as in \eqref{L}
denotes the linearized operator about $\bar U$ and
$\cN$ as in \eqref{N1} is a nonlinear residual, satisfying
\ba\label{Nbds}
\cN(V)&=O(|V|^2+|V||V_x|).\\
\partial_x \cN(V)&=O(|V||\partial_x V| + |V_x|^2+|V||V_{xx}|)\\
%\partial_x^2 N(v)&= O(|\partial_x|^2+ |v||\partial_x^2v|).\\
\ea
so long as $|V|_{H^1}$ (hence $|V|_{L^\infty}$ and $|U|_{L^\infty}$) 
remains bounded. 

\subsection{Projector bounds}\label{projbds}
Let $\Pi_u$ denote the eigenprojection of $L$ onto its
unstable subspace $\Sigma_u$, and $\Pi_{cs}={\rm Id}- \Pi_u$
the eigenprojection onto its center stable subspace $\Sigma_{cs}$. 

\bl\label{projlem}
Assuming (H0)--(H1),
\be\label{comm}
\Pi_j \partial_x= \partial_x \tilde \Pi_j
\ee
for $j=u,\, cs$ and, for all $1\le p\le \infty$, $0\le r\le 4$, 
\ba
|\Pi_{u}|_{W^{r,p}\to W^{r,p}}, |\tilde \Pi_{u}|_{W^{r,p}\to W^{r,p}}&\le C,\\
|\tilde \Pi_{cs}|_{W{r,p}\to W^{r,p} }, 
\; |\tilde \Pi_{cs}|_{W{r,p}\to W^{r,p}}&\le C.
\ea
\el

\begin{proof}
Recalling (see the proof of Theorem \ref{t:maincs})
that $L$ has at most finitely many unstable eigenvalues,
we find that $\Pi_u$ may be expressed as
$$
\Pi_u f= \sum_{j=1}^p \phi_j(x) \langle \tilde \phi_j, f\rangle,
$$
where $\phi_j$, $j=1, \dots p$ are generalized right eigenfunctions of
$L$ associated with unstable eigenvalues $\lambda_j$, 
satisfying the generalized eigenvalue equation $(L-\lambda_j)^{r_j}\phi_j=0$,
$r_j\ge 1$, and $\tilde \phi_j$
are generalized left eigenfunctions.
Noting that $L$ is divergence form, and that $\lambda_j\ne 0$,
we may integrate $(L-\lambda_j)^{r_j}\phi_j=0$ over $\R$ to
obtain $\lambda_j^{r_j}\int \phi_j dx=0$ and thus $\int\phi_j dx=0$.
Noting that $\phi_j$, $\tilde \phi_j$ and derivatives decay exponentially
by standard theory \cite{He,ZH,MaZ1}, we find that
$$
\phi_j= \partial_x \Phi_j
$$
with $\Phi_j$ and derivatives exponentially decaying, hence
$$
\tilde \Pi_u f=\sum_j \Phi_j \langle \partial_x \tilde \phi, f\rangle.
$$
Estimating 
$$
|\partial_x^j\Pi_u f|_{L^p}=|\sum_j \partial_x^j\phi_j \langle \tilde \phi_j f
\rangle|_{L^p}\le
\sum_j |\partial_x^j\phi_j|_{L^p} |\tilde \phi_j|_{L^q} |f|_{L^p}
\le C|f|_{L^p}
$$
for $1/p+1/q=1$
and similarly for $\partial_x^r \tilde \Pi_u f$, we obtain the claimed
bounds on $\Pi_u$ and $\tilde \Pi_u$, from which the bounds on
$\Pi_{cs}={\rm Id}-\Pi_u$ and
$\tilde \Pi_{cs}={\rm Id}-\tilde \Pi_u$ follow immediately.
\end{proof}

\subsection{Linear estimates}

Let $G_{cs}(x,t;y):=\Pi_{cs}e^{Lt}\delta_y(x)$ denote
the Green kernel of the linearized solution operator on
the center stable subspace $\Sigma_{cs}$.
Then, we have the following detailed pointwise bounds
established in \cite{TZ2,MaZ1}.

\begin{proposition}[\cite{TZ2,MaZ1}]\label{greenbounds}
Under (A0)--(A3), (H0)--(H3), (D1)--(D3), 
the center stable Green function may be decomposed as 
%\be\label{Gdecomp}
$G_{cs}=E+ H + \tilde G,$
%\ee
where 
\begin{equation}\label{E}
E(x,t;y)= \d_x \bar U(x) e_j(y,t),
\end{equation}
\begin{equation}\label{e}
  e(y,t)=\sum_{a_k^{-}>0}
  \left(\textrm{errfn }\left(\frac{y+a_k^{-}t}{\sqrt{4t}}\right)
  -\textrm{errfn }\left(\frac{y-a_k^{-}t}{\sqrt{4t}}\right)\right)
  l_{k}^{-}(y)
\end{equation}
for $y\le 0$ and symmetrically for $y\ge 0$, 
$l_k^-\in \R^n$ constant, 
\be\label{H}
H(x,t;y)=
\sum_{j=1}^{J} 
{O}(e^{-\theta_0 t}) \delta_{x-\bar a_j^* t}(-y) ,
\quad \theta>0,
\ee
and
\begin{equation}\label{Gbounds}
\begin{aligned}
|\tilde G(x,t;y)|&\le  Ce^{-\eta(|x-y|+t)} +
\sum_{k=1}^n t^{-1/2}e^{-(x-y-a_k^{-} t)^2/Mt} e^{-\eta x^+} \\
&+
\sum_{a_k^{-} > 0, \, a_j^{-} < 0} 
\chi_{\{ |a_k^{-} t|\ge |y| \}}
t^{-1/2} e^{-(x-a_j^{-}(t-|y/a_k^{-}|))^2/Mt}
e^{-\eta x^+} \\
&+
\sum_{a_k^{-} > 0, \, a_j^{+}> 0} 
\chi_{\{ |a_k^{-} t|\ge |y| \}}
t^{-1/2} e^{-(x-a_j^{+} (t-|y/a_k^{-}|))^2/Mt}
e^{-\eta x^-}, \\
\end{aligned}
\end{equation}
\begin{equation}\label{Gybounds}
\begin{aligned}
|\partial_y \tilde G(x,t;y)|&\le  Ce^{-\eta(|x-y|+t)}
+ Ct^{-1/2}\Big( \sum_{k=1}^n 
t^{-1/2}e^{-(x-y-a_k^{-} t)^2/Mt} e^{-\eta x^+} \\
&+
\sum_{a_k^{-} > 0, \, a_j^{-} < 0} 
\chi_{\{ |a_k^{-} t|\ge |y| \}}
t^{-1/2} e^{-(x-a_j^{-}(t-|y/a_k^{-}|))^2/Mt}
e^{-\eta x^+} \\
&+
\sum_{a_k^{-} > 0, \, a_j^{+}> 0} 
\chi_{\{ |a_k^{-} t|\ge |y| \}}
t^{-1/2} e^{-(x-a_j^{+} (t-|y/a_k^{-}|))^2/Mt}
e^{-\eta x^-}\Big) \\
\end{aligned}
\end{equation}
for $y\le 0$ and symmetrically for $y\ge 0$,
for some $\eta$, $C$, $M>0$, where 
$a_j^\pm$ are the eigenvalues of $A_\pm=dF(U_\pm)$, 
$a_j^*$, $j=1,\dots, J$ are the eigenvalues of $A_{11}$,
$x^\pm$ denotes the positive/negative
part of $x$, and  indicator function $\chi_{\{ |a_k^{-}t|\ge |y| \}}$ is 
$1$ for $|a_k^{-}t|\ge |y|$ and $0$ otherwise.
\end{proposition}

\begin{proof}
As observed in \cite{TZ2},
it is equivalent to establish decomposition 
\be\label{fulldecomp}
G=G_u + E+ H+ \tilde G
\ee
for the full Green function $G(x,t;y):=e^{Lt}\delta_y(x)$,
where 
$$
G_u(x,t;y):=\Pi_u e^{Lt}\delta_y(x)
=
e^{\gamma t}\sum_{j=1}^p\phi_j(x)\tilde \phi_j(y)^t
$$
for some constant matrix $M\in \C^{p\times p}$
denotes the Green kernel of the linearized solution operator
on $\Sigma_u$, $\phi_j$ and $\tilde\phi_j$ right and left
generalized eigenfunctions associated with unstable eigenvalues
$\lambda_j$, $j=1,\dots,p$.

The problem of describing the full Green function
has been treated in \cite{ZH, MaZ3}, 
starting with the Inverse Laplace Transform representation
\be\label{ILT2}
G(x,t;y)=e^{Lt}\delta_y(x)= \oint_\Gamma e^{\lambda t}(\lambda-L(\e))^{-1} 
\delta_y(x)d\lambda \, ,
\ee
where 
$$
\Gamma:= \partial \{ \lambda : \Re \lambda\le \eta_1 - \eta_2 |\Im \lambda|\}
$$
is an appropriate sectorial contour, $\eta_1$, $\eta_2>0$;
estimating the resolvent kernel 
$G^\eps_\lambda(x,y):=(\lambda-L(\e))^{-1}\delta_y(x)$
using Taylor expansion in $\lambda$,
asymptotic ODE techniques in $x$, $y$, and judicious decomposition
into various scattering, excited, and residual modes;
then, finally, estimating the contribution of various modes to \eqref{ILT2}
by Riemann saddlepoint (Stationary Phase) method, moving contour
$\Gamma$ to a optimal, ``minimax'' positions for each
mode, depending on the values of $(x,y,t)$.

In the present case, we may first move $\Gamma$ to a contour
$\Gamma'$ enclosing (to the left) all spectra of $L$
except for the $p$ unstable eigenvalues $\lambda_j$, $j=1, \dots, p$,
to obtain
$$
G(x,t;y)= \oint_{\Gamma'} e^{\lambda t}(\lambda-L)^{-1} d\lambda
+ \sum_{j=\pm} 
\Res_{\lambda_j(\eps)} \big( e^{\lambda t}(\lambda-L)^{-1}
\delta_y(x) \big),
$$
where
$\Res_{\lambda_j(\eps)} \big( e^{\lambda t}(\lambda-L)^{-1}
\delta_y(x) \big) = G_u(x,t;y)$, then estimate the remaining term
$ \oint_{\Gamma'} e^{\lambda t}(\lambda-L)^{-1} d\lambda$
on minimax contours as just described.
See the proof of Proposition 7.1, \cite{MaZ3}, for a detailed 
discussion of minimax estimates $E+G$ and of Proposition 7.7, 
\cite{MaZ3} for a complementary discussion of residues
incurred at eigenvalues in $\{\Re \lambda\ge 0\}\setminus\{0\}$.
See also \cite{TZ1}.
\end{proof}

\bc [\cite{MaZ1}] \label{lpbds}
Assuming (A1)--(A3), (H0)--(H3), (D1)--(D3),
\be \label{Hbounds}
|\int_{-\infty}^{+\infty} H(\cdot,t;y)f(y)dy|_{L^p}
\le C e^{-\theta t} |f|_{L^p},
\ee
\be \label{tGbounds}
|\int_{-\infty}^{+\infty} \tG(\cdot,t;y)f(y)dy|_{L^p}
\le C (1+t)^{-\frac{1}{2}(\frac{1}{q}-\frac{1}{p})} |f|_{L^q},
\ee
\be\label{tGybounds}
|\int_{-\infty}^{+\infty} \tG_y(\cdot,t;y)f(y)dy|_{L^p}
\le C (1+t)^{-\frac{1}{2}(\frac{1}{q}-\frac{1}{p})-\frac{1}{2}} |f|_{L^q},
\ee
for all $t\ge 0$, some $C>0$, for any
$1\le q\le p$ (equivalently, $1\le r\le p$)
and $f\in L^q$, where $1/r+1/q=1+1/p$.
\ec

\begin{proof}
Standard convolution inequalities together
with bounds \eqref{H}--\eqref{Gybounds}; see \cite{MaZ1,MaZ2,MaZ3,Z2} 
for further details.
\end{proof}

\bc[\cite{Z4}]\label{ebds}
The kernel ${e}$ satisfies
$$
|{e}_y (\cdot, t)|_{L^p},  |{e}_t(\cdot, t)|_{L^p} 
\le C t^{-\frac{1}{2}(1-1/p)},
\label{36}
$$
$$
|{e}_{ty}(\cdot, t)|_{L^p} 
\le C t^{-\frac{1}{2}(1-1/p)-1/2},
\label{37}
$$
for all $t>0$.  
%%%%%
%BEGIN CHANGE, March 24
Moreover, for $y\le 0$ we have the pointwise bounds
$$
|{e}_y (y,t)|, |{e}_t (y,t)| 
\le 
Ct^{-\frac{1}{2}} \sum_{a_k^->0}
\Big(e^{-\frac{(y+a_k^-t)^2}{Mt}}+
e^{-\frac{(y-a_k^-t)^2}{Mt}}\Big),
\label{38}
$$
$$
|{e}_{ty} (y,t)| \le 
Ct^{-1} \sum_{a_k^->0}
\Big(e^{-\frac{(y+a_k^-t)^2}{Mt}}+
e^{-\frac{(y-a_k^-t)^2}{Mt}}\Big),
\label{39}
$$
for $M>0$ sufficiently large, 
and symmetrically for $y\ge 0$.
\ec

\begin{proof}
Direct computation using definition \eqref{e}; 
see \cite{Z4,MaZ2,MaZ3} or \cite{Z5}, Appendix A.
\end{proof}

%NOTE: to get higher $x$-derivs, differentiate equations and keep
%principal part plus lower-order commutator terms, proceed by iteration.
%details ommitted... (no problem here in fact...)

\subsection{Reduced equations II}
%(instantaneous projection)
Recalling that $\d_x\bar U$ is a stationary
solution of the linearized equations $U_t=LU$,
so that $L\d_x\bar U =0$, or
$$
\int^\infty_{-\infty}G(x,t;y)\bar U_x(y)dy=e^{Lt}\bar U_x(x)
=\d_x\bar U(x),
$$
we have, applying Duhamel's principle to \eqref{nlpert},
$$
\begin{array}{l}
  \displaystyle{
  V(x,t)=\int^\infty_{-\infty}G(x,t;y)V_0(y)\,dy } \\
  \displaystyle{\qquad
  -\int^t_0 \int^\infty_{-\infty} G_y(x,t-s;y)
  (\cN(V)+\dot \alpha V ) (y,s)\,dy\,ds + \alpha (t)\d_x \bar u(x).}
\end{array}
$$
Defining 
\begin{equation}
 \begin{array}{l}
  \displaystyle{
  \alpha (t)=-\int^\infty_{-\infty}e(y,t) V_0(y)\,dy }\\
  \displaystyle{\qquad
  +\int^t_0\int^{+\infty}_{-\infty} e_{y}(y,t-s)(\cN(V)+
  \dot \alpha\, V)(y,s) dy ds, }
  \end{array}
 \label{alpha}
\end{equation}
following \cite{ZH,Z4,MaZ2,MaZ3}, 
where $e$ is as in \eqref{e}, 
and recalling the decomposition $G=E+ H+ G_u+ \tilde G$ of \eqref{fulldecomp},
we obtain the {\it reduced equations}
\ba \label{v}
  V(x,t)&=\int^\infty_{-\infty} (G_u+H+\tilde G)(x,t;y)V_0(y)\,dy \\
  &\quad -\int^t_0\int^\infty_{-\infty}(G_u+\tilde G)_y(x,t-s;y)(\cN(V)+
  \dot \alpha V)(y,s) dy \, ds \\
  &\quad +\int^t_0\int^\infty_{-\infty}H(x,t-s;y)(\cN(V)+
  \dot \alpha V)_y(y,s) dy \, ds, \\
\ea
and, differentiating (\ref{alpha}) with respect to $t$,
and observing that 
$e_y (y,s)\rightharpoondown 0$ as $s \to 0$, as the difference of 
approaching heat kernels,
\begin{equation}
 \begin{array}{l}
 \displaystyle{
  \dot \alpha (t)=-\int^\infty_{-\infty}e_t(y,t) V_0(y)\,dy }\\
 \displaystyle{\qquad
  +\int^t_0\int^{+\infty}_{-\infty} e_{yt}(y,t-s)(\cN(V)+
  \dot \alpha V)(y,s)\,dy\,ds. }
 \end{array}
\label{alphadot}
\end{equation}
\medskip

Note that this (nonlocal in time) choice of $\alpha$ 
and the resulting reduced equations are different from those
of Section \ref{s:trans}.
As discussed further in \cite{Go,Z4,MaZ2,MaZ3,Z2},
$\alpha$ may be considered in the present context as defining a notion of
approximate shock location.

\subsection{Nonlinear damping estimate}

\begin{proposition}[\cite{MaZ3}]\label{damping}
Assuming (A0)-(A3), (H0)-(H2), let $V_0\in H^{3}$, 
and suppose that for $0\le t\le T$, the $H^{3}$ norm of $V$
remains bounded by a sufficiently small constant, for $V$ as in
\eqref{pert2} and $u$ a solution of \eqref{cons}.
Then, for some constants $\theta_{1,2}>0$, for all $0\leq t\leq T$,
\begin{equation}\label{nlEbounds}
\|V(t)\|_{H^3}^2 \leq C e^{-\theta_1 t} \|V(0)\|^2_{H^3} 
+ C \int_0^t e^{-\theta_2(t-s)} (|V|_{L^2}^2 + |\dot \alpha|^2) (s)\,ds.
\end{equation}
\end{proposition}

\begin{proof}
Essentially identical to the proof of Proposition \ref{localdamp}
and corollaries, but using Moser's inequality to bound the nonlinear
term instead of the Lipshitz bound imposed by truncation.
See \cite{MaZ4,Z2,Z3} for detailed proofs of more general results.
\end{proof}

\subsection{Proof of nonlinear stability}

Decompose now the nonlinear perturbation $V$ as
\be\label{vdecomp}
V(x,t)=w(x,t)+z(x,t),
\ee
where
\be\label{wzdef}
w:=\Pi_{cs}V, \quad z:=\Pi_u V.
\ee
Applying $\Pi_{cs}$ to \eqref{v} and recalling commutator
relation \eqref{comm}, we obtain an equation
\ba \label{w}
  w(x,t)&=\int^\infty_{-\infty} \tilde G (x,t;y)w_0(y)\,dy \\
  &\quad -\int^t_0\int^\infty_{-\infty} \tilde G_y (x,t-s;y)
\tilde \Pi_{cs}(\cN(V)+
  \dot \alpha V)(y,s) dy \, ds\\
  &\quad +\int^t_0\int^\infty_{-\infty} H (x,t-s;y)
\Pi_{cs} (\cN(V)+
  \dot \alpha V)_y (y,s) dy \, ds
\ea
for the flow along the center stable manifold, parametrized by
$w\in \Sigma_{cs}$.

\bl\label{quadlem} Assuming (A1)--(A3), (H0)--(H2), for $V$ lying initially
on the center stable manifold $\cM_{cs}$,
\be\label{vwbd}
|z|_{W^{r,p}}\le C|w|_{H^2}^2
\ee 
for some $C>0$, for all $1\le p\le \infty$ and $0\le r\le 4$,
so long as $|w|_{H^2}$ remains sufficiently small.
\el

\begin{proof}
By \eqref{tanbd1}, we
have immediately $|z|_{H^{1,2}}\le C|w|_{H^{1,2}}^2$, whence
\eqref{vwbd} follows by equivalence of norms for finite-dimensional
vector spaces, applied to the $p$-dimensional subspace $\Sigma_u$.
(Alternatively, we may see this by direct computation using
the explicit description of $\Pi_u V$ afforded by Lemma \ref{projlem}.)
\end{proof}

\begin{proof}[Proof of Theorem \ref{t:mainstab}]

%instab. part:
Recalling by Theorem \ref{t:maincs}
that solutions remaining for all time in a sufficiently
small radius neighborhood $\cN$ of the set of translates of $\bar u$
lie in the center stable manifold $\cM_{cs}$, we obtain trivially
that solutions not originating in $\cM_{cs}$ must exit $\cN$ in finite time,
verifying the final assertion of orbital instability with respect
to perturbations not in $\cM_{cs}$.

Consider now a solution $V \in \cM_{cs}$, or, equivalently,
a solution $w\in \Sigma_{cs}$ of \eqref{w} with 
$z=\Phi_{cs}(w)\in \Sigma_u$.
Define
\begin{equation}
\label{zeta2}
 \zeta(t):= \sup_{0\le s \le t}
 \Big( |w|_{H^3}(1+s)^{\frac{1}{4}} + 
(|w|_{L^\infty}+ |\dot \alpha (s)|)(1+s)^{\frac{1}{2}} \Big).
\end{equation}
We shall establish:

{\it Claim.} For all $t\ge 0$ for which a solution exists with
$\zeta$ uniformly bounded by some fixed, sufficiently small constant,
there holds
\begin{equation}
\label{claim}
\zeta(t) \leq C_2(E_0 + \zeta(t)^2)
\quad \hbox{\rm for} \quad
E_0:=|V_0|_{L^1\cap H^3}.
\end{equation}
\medskip

{}From this result, provided $E_0 < 1/4C_2^2$, 
we have that $\zeta(t)\le 2C_2E_0$ implies
$\zeta(t)< 2C_2E_0$, and so we may conclude 
by continuous induction that
 \begin{equation}
 \label{bd}
  \zeta(t) < 2C_2E_0
 \end{equation}
for all $t\geq 0$, 
%whence we obtain the stated bounds by definition \eqref{zeta2}
%together with \eqref{vwbd}.
from which we readily obtain the stated bounds.
%NOTE: for example, the bounds in $H^s$ and $L^p$, $p\ge 2$ 
%follow by definition \eqref{zeta2} and $L^p$ interpolation.
(By standard short-time $H^s$ existence theory, 
$V\in H^3$ exists and $\zeta$ remains
continuous so long as $\zeta$ remains bounded by some uniform constant,
hence \eqref{bd} is an open condition.)
\medskip

{\it Proof of Claim.}
By Lemma \ref{quadlem},
$$
|w_0|_{L^1\cap H^3}\le |V_0|_{L^1\cap H^3}+ |z_0|_{L^1\cap H^3}
\le
|V_0|_{L^1\cap H^3}+ C|w_0|_{H^3}^2, 
$$
whence
$$
|w_0|_{L^1\cap H^3}\le CE_0.
$$
Likewise, by Lemma \ref{quadlem},
\eqref{zeta2}, \eqref{Nbds}, and Lemma \ref{projlem}, for $0\le s\le t$
and $2\le p\le \infty$,
\ba\label{Nlast}
|\tilde \Pi_{cs}(N(V)+ \dot \alpha V)(y,s)|_{L^2}&\le C\zeta(t)^2 
(1+s)^{-\frac{3}{4}},\\
| \Pi_{cs}(N(V)+ \dot \alpha V)_y(y,s)|_{L^p}&\le C\zeta(t)^2 
(1+s)^{-\frac{1}{2}},\\
\ea

Combining the latter bounds with representations \eqref{w}--\eqref{alphadot}
and applying Corollary \ref{lpbds}, we obtain
 \ba\label{claimw}
  |w(x,t)|_{L^p} &\le
  \Big|\int^\infty_{-\infty} \tilde G(x,t;y) w_0(y)\,dy\Big|_{L^p}
   \\
 &\qquad +
\Big|\int^t_0
  \int^\infty_{-\infty} \tilde G_y(x,t-s;y)  \tilde \Pi_{cs}(N(V)+
  \dot \alpha V)(y,s)  dy \, ds\Big|_{L^p} \\
 &\qquad +
\Big|\int^t_0
  \int^\infty_{-\infty} H(x,t-s;y)  \Pi_{cs}(N(V)+
  \dot \alpha V)_y(y,s)  dy \, ds\Big|_{L^p} \\
  & \le
  E_0 (1+t)^{-\frac{1}{2}(1-\frac{1}{p})}
 + C\zeta(t)^2 \int^t_0 (t-s)^{-\frac{3}{4}+\frac{1}{2p}}
(1+s)^{-\frac{3}{4}} dy \, ds \\
 &\quad +
C\zeta(t)^2 \int^t_0 e^{-\theta (t-s)} (1+s)^{-\frac{1}{2}} dy \, ds \\
&\le
C(E_0+\zeta(t)^2)(1+t)^{-\frac{1}{2}(1-\frac{1}{p})}
\ea
and, similarly, using H\"older's inequality and
applying Corollary \ref{ebds},
\ba\label{claimalpha}
 |\dot \alpha(t)| &\le \int^\infty_{-\infty}|e_t(y,t)|
  |V_0(y)|\,dy \\
  &\qquad +\int^t_0\int^{+\infty}_{-\infty} |e_{yt}(y,t-s)||\tilde \Pi_{cs}
(N(V)+ \dot \alpha V)(y,s)|\,dy\,ds\\
&\le |e_t|_{L^\infty} |V_0|_{L^1}
+ C\zeta(t)^2 \int^t_0
|e_{yt}|_{L^2}(t-s) |\tilde \Pi_{cs}(N(V)+ \dot \alpha V)|_{L^2}(s) ds\\
&\le E_0 (1+t)^{-\frac{1}{2}}
+ C\zeta(t)^2 \int^t_0
(t-s)^{-\frac{3}{4}}(1+s)^{-\frac{3}{4}} ds\\
&\le C(E_0+\zeta(t)^2)(1+t)^{-\frac{1}{2}}.\\
\ea
Applying Lemma \ref{damping} and using \eqref{claimw} and
\eqref{claimalpha},
we obtain, finally,
\be\label{claimwH2}
|w|_{H^3}(t)\le C(E_0+\zeta(t)^2)(1+t)^{-\frac{1}{4}}.
\ee
Combining \eqref{claimw}, \eqref{claimalpha}, and \eqref{claimwH2},
we obtain \eqref{claim} as claimed.

\medbreak
As discussed earlier,
from \eqref{claim}, we obtain by continuous induction \eqref{bd}, or
$
 \zeta\le 2C_2|V_0|_{L^1\cap H^2}, 
$
whereupon the claimed bounds on $|V|_{L^p}$ and $|V|_{H^3}$ follow by
\eqref{claimw} and \eqref{claimwH2}, and on $|\dot \alpha|$ 
by \eqref{claimalpha}.
Finally, a computation parallel to \eqref{claimalpha} 
(see, e.g., \cite{MaZ3,Z2}) 
yields $| \alpha(t)| \le C(E_0+\zeta(t)^2)$, from which
we obtain the last remaining bound on $|\alpha(t)|$.
\end{proof}

\medbreak
{\bf Acknowledgement.}
Thanks to Milena Stanislavova and Charles Li for two interesting discussions
that inspired this work.
% and to Milena Stanislavova for pointing out the reference in \cite{GJLS}.

%%%%%%%%%%%%%%%%%%%

%APPENDIX

%%%%%%%%%%%%%%
%
\appendix

\section{Example systems}\label{examples}

1. The general Navier--Stokes equations of compressible gas dynamics, written in
Lagrangian coordinates, appear as
\begin{equation}
\left\{\begin{array}{l}
v_t -u_x = 0,\\
u_t + p_x =((\nu/v) u_{x})_x,\\
(e+u^2/2)_t + (pu)_x = ((\kappa/v) T_x + (\mu/v) uu_x)_x,
\end{array}\right.
\label{NS}
\end{equation}
where $v>0$ denotes specific volume, $u$ velocity, $e>0$ internal energy, $T=T(v,e)>0$
temperature, %$T=c^{-1}e>0$ temperature, 
$p=p(v,e)$ pressure, and $\mu>0$ and $\kappa>0$
are coefficients of viscosity and heat conduction, respectively.
Defining $U_1:=(v)$, $U_2:=(u,e+|u|^2/2)$, we find that
conditions (A1)--(A3) and (H1)--(H3)
are thus satisfied under the mild assumptions of ideal temperature dependence
\be\label{ideal}
T=T(e)
\ee
monotone temperature-dependence
\be\label{Tstab}
T_e>0,
\ee
and thermodynamic stability of the endstates, 
\be\label{thermstab}
(p_v)_\pm <0,  (T_e)_\pm >0;
\ee
see, e.g., \cite{MaZ4,Z2,Z3,TZ3} for further discussion.
Notably, this allows the interesting case of a van der Waals-type equation of state,
with $p_v>0$ for some values of $v$ along
the connecting profile.

2. The equations of MHD in Lagrangian coordinates are
\begin{equation}
\left\{\begin{array}{l}
v_t -u_{1x} = 0,\\
u_{1t} + (p+ (1/2\mu_0)(B_2^2+B_3^2))_x =((\nu/v) u_{1x})_x,\\
u_{2t}  - ((1/\mu_0)B_1^*B_2)_x =((\nu/v) u_{2x})_x,\\
u_{3t}  - ((1/\mu_0)B_1^*B_3)_x =((\nu/v) u_{3x})_x,\\
(vB_2)_{t}  - (B_1^*u_2)_x =((1/\sigma\mu_0 v) B_{2x})_x,\\
(vB_3)_{t}  - (B_1^*u_3)_x =((1/\sigma\mu_0 v) B_{3x})_x,\\
(e+(1/2)(u_1^2+ u_2^2+u_3^2) + (1/2\mu_0)v(B_2^2+B_3^2))_t \\
\qquad\qquad + [(p+ (1/2\mu_0)(B_2^2+B_3^2))u_1 - (1/\mu_0)B_1^*
(B_2u_2+B_3u_3)]_x \\
\qquad\quad
= [(\nu/v)u_1u_{1x} + (\mu/v) (u_2u_{2x} + u_3 u_{3x})\\
\qquad\qquad +(\kappa/v) T_x +  (1/\sigma \mu_0^2 v)(B_2B_{2x} + B_3 B_{3x})]_x,
\end{array}\right.
\label{MHD}
\end{equation}
where $v$ denotes specific volume, $u=(u_1,u_2,u_3)$ velocity,
$p=P(v,e)$ pressure, $B=(B_1^*,B_2,B_3)$ magnetic induction,
$B_1^*$ constant, $e$ internal energy, $T=T(v,e)>0$ temperature, and
$\mu>0$ and $\nu>0$ the two coefficients of viscosity, $\kappa>0$
the coefficient of heat conduction, $\mu_0>0$ the magnetic
permeability, and $\sigma>0$ the electrical resistivity.
%Under assumptions (\ref{ideal}), (\ref{Tstab}), and (\ref{thermstab}),
Under (\ref{ideal})--(\ref{thermstab}),
conditions (A1)--(A3) are again satisfied, and
conditions (H1)--(H3) are satisfied (see \cite{MaZ4}) under the generically satisfied
assumptions that the endstates $U^\eps_\pm$
be strictly hyperbolic (i.e., have simple eigenvalues), 
and the speed $s$ be nonzero, i.e., the shock move with nonzero
speed relative to the background fluid velocity, with $U_1:=(v)$, $U_2:=
(u, B, e+|u|^2/2 + v|B|^2/2\mu_0)$.
(For gas dynamics, only Lax-type shocks and nonzero speeds can occur,
and all points $U$ are strictly hyperbolic.)

3. (MHD with infinite resistivity/permeability)
An interesting variation of (\ref{MHD}) that is of interest in certain astrophysical
parameter regimes is the limit in which either electrical resistivity $\sigma$, magnetic
permeability $\mu_0$, or both, go to infinity,
in which case the righthand sides of the fifth and sixth equations
of (\ref{MHD}) go to zero and there is a three-dimensional
set of hyperbolic modes $(v,vB_2,vB_3)$ instead of the usual one.
By inspection, the associated equations are still linear in the
conservative variables.
Likewise, (A1)--(A3), (H1)--(H3) hold under (\ref{ideal})--(\ref{thermstab})
for nonzero speed shocks with strictly hyperbolic
endstates. 

4. (multi-species gas dynamics or MHD)
Another simple example for which the hyperbolic modes are vectorial
is the case of miscible, multi-species flow, neglecting species diffusion,
in either gas dynamics or magnetohydrodynamics.
In this case, the hyperbolic modes consist of $k$ copies of the hyperbolic
modes for a single species, where $k$ is the number of total species.
Again, (A1)--(A3), (H1)--(H3) hold for nonzero
speed shocks with strictly hyperbolic endstates
under assumptions (\ref{ideal})--(\ref{thermstab}).

%%%%%%%%%%%%%%

\end{document}